\documentclass[a4paper]{article}
\usepackage[latin1]{inputenc}
\usepackage[T1]{fontenc}
\usepackage[english]{babel}
\usepackage{amsfonts}
\usepackage{amsmath}
\usepackage{enumerate}
\usepackage{amssymb}
\usepackage{stmaryrd}
\usepackage{hyperref}
\usepackage{amscd}
\usepackage{amsthm}
\usepackage[left=3cm,right=3cm,top=4cm,bottom=4cm]{geometry}
\usepackage{fontenc}
\usepackage{cancel}
\usepackage[pdftex]{graphicx}
\usepackage{xcolor}
\DeclareMathOperator*{\argmin}{arg\,min}
\usepackage{mathrsfs}
\newcommand{\ls}{\leqslant}
\newcommand{\gs}{\geqslant}

\newcommand{\R}{\mathbb R}
\newcommand{\Rn}{\mathbb R^n}

\newcommand{\eps}{\varepsilon}
\newcommand{\tr}{\operatorname{Tr}}
\newcommand{\per}{\operatorname{Per}}
\renewcommand{\div}{\operatorname{div}}
\newcommand{\scal}[2]{\left \langle #1 \, , \, #2 \right \rangle}
\newtheorem{thm}{Theorem}
\newtheorem{prop}{Proposition}
\newtheorem{lem}{Lemma}
\newtheorem{cor}{Corollary}
\theoremstyle{definition}
\newtheorem{defi}{Definition}
\theoremstyle{remark}
\newtheorem{rem}{Remark}
\author{Gwenael Mercier\footnote{RICAM, Austrian Academy of Sciences: \texttt{gwenael.mercier@ricam.oeaw.ac.at}}}
\title{Mean curvature flow with obstacles: a viscosity approach}
\date{}

\begin{document}
\maketitle
\begin{abstract}
 We introduce a level-set formulation for the mean curvature flow with obstacles and show existence and uniqueness of a viscosity solution. These results generalize a well known viscosity approach for the mean curvature flow without obstacle by Evans and Spruck and Chen, Giga and Goto in 1991. In addition, we show that this evolution is consistent with the variational scheme introduced by Almeida, Chambolle and Novaga (2012) and we study the long time behavior of our viscosity solutions.
 
\paragraph{Classification:} 53C44
\end{abstract}

\section{Introduction}
In this article, we introduce the level set formulation for a generalized motion by mean curvature with obstacles. More precisely, let $M(t) = \partial E(t)$ be a family of $n-1$ submanifold in $\Rn$, we say that it evolves by mean curvature if for any $x \in M(t)$, the velocity of $M(t)$ at $x$ is given by
\begin{equation}v(x) = - H \nu(x) \label{eqgeom}\end{equation} where $H$ is the mean curvature of $M(t)$ at $x$ (nonnegative if $E(t)$ is a convex set with boundary) and $\nu$ is the normal vector to $M(t)$ pointing towards $E(t)^c$.

Motivated by recent works from Almeida, Chambolle and Novaga \cite{chambolle12} and Spadaro \cite{spadaro11} about a discrete scheme for the mean curvature flow with obstacles, we want to constrain \eqref{eqgeom} forcing
\begin{equation} \Omega^-(t) \subset E(t) \subset \Omega^+(t) \label{congeom}\end{equation} where $\Omega^\pm$ are two open sets (which can depend on the time variable).

Mean curvature flow has been widely studied in the 30 past years for physical and biological purposes. For instance, one can mention \cite{almeida09, almeida11} for a new model in biology (tissue repair) using this evolution. Concerning the mathematical study, one can in particular cite \cite{brakke78} for a first paper on this motion, \cite{ecker91} for a geometric study of \eqref{eqgeom} and \cite{evans91} and \cite{chen91} for a level-set formulation and the use of viscosity solutions. In the sequel we follow the last approach.

It is well known (see for example \cite{evans91}) that if $u : \Rn \to \R$ is a smooth function with a nonzero gradient at $x_0$, the mean curvature of the level set $\{u=u(x_0)\}$ at $x_0$ is given by $\div\left( \frac{Du}{|Du|} \right) (x_0).$ As a result, making this set (and every other level-set of $u$) evolve by mean curvature leads to the following equation for $u$:
\begin{equation}
u_t = |Du| \div \left( \frac{Du}{|Du|} \right).
\label{mcflow}
\end{equation}
In the whole paper, we will think of $M(t)$ as the zero-level-set of $u(\cdot,t)$.\\
To add the constraint to \eqref{mcflow}, we define $u^\pm(x,t)$ such that
$$\Omega^-(t) \subset E(t) \subset \Omega^+(t) \Leftrightarrow \{u^+ < 0\} \subset \{u < 0\} \subset \{u^- < 0\}$$
and impose
\begin{equation} \forall x,t, \quad u^-(x,t) \ls u(x,t) \ls u^+(x,t).\label{viscon}\end{equation}

As in \cite{evans91}, \cite{chen91}, we study \eqref{mcflow} with constraint \eqref{viscon} using viscosity solutions. We first present a suitable viscosity framework and prove a uniqueness and existence result for bounded uniformly continuous initial data and obstacles and Lipschitz forcing term in the spirit of \cite{c92user}. Then, we link the regularity of the solution to the regularity of the initial data. 

We also show that our level-set approach really defines a geometric flow: the $\alpha$-level set of the solution depends only on the $\alpha$-level set of the initial data and the obstacles. Nonetheless, as expected, there is no real geometrical uniqueness: level sets of the solution can develop non empty interiors because of the obstacles (even if the free evolution does not). In an upcoming paper with Matteo Novaga \cite{mernov}, we study the MCF with obstacles with a geometrical point of view (in the spirit of \cite{ecker91}), proving short time existence, uniqueness and regularity of solutions.

Finally, in Section \ref{longtime}, we compare the approach followed by \cite{spadaro11} and \cite{chambolle12} (discrete minimizing scheme based on \cite{ATW}) to ours. More precisely, we show that the discrete scheme has a limit which is the viscosity solution to \eqref{mcflow} with constraint \eqref{viscon}. In addition, this variational approach gives monotonicity of the flow and therefore information on the long time behavior of the viscosity solution.

\section{Notation}
In what follows, we consider the equation (slightly more general than \eqref{mcflow}, but the latter has to be kept in mind), for $u : \Rn \times \R^+ \to \R$
\begin{equation}
 \forall t \gs 0, x\in \mathbb R^n, \quad u_t +F(Du,D^2u)+k|Du| =0, \label{motion}
\end{equation}
where $k : \Rn \times \R ^+ \to \R$ is a forcing term and $F : \Rn \times \mathcal S_n \to \R$ ($\mathcal S_n$ is the set of symmetric matrices of dimension $n$) satisfies
\begin{enumerate}[i)]
 \item $F \in \mathcal C\left( \R^n \setminus \{0\} \times \mathcal S_n(\R) \right)$ ,
\item $F$ is geometric : $\forall \lambda >0, \sigma \in \R, \; F(\lambda p,\lambda X+\sigma p \otimes p)=\lambda F(p,X)$,
\item For $X$ and $Y$ symmetric matrices with $X\ls Y$, $F(p,X) \ls F(p,Y)$.
\end{enumerate}

In the following, $Du$ and $D^2 u$ denote space derivatives only.\\
We will denote by $u \wedge v$ and $u \vee v$ the quantities $\min(u;v)$ and $\max(u;v)$.

We also introduce $F^\ast $ and $F_\ast $ which are respectively the upper semicontinuous and lower semicontinuous envelopes of $u$\footnote{This quantity is useful to make the following results apply for the mean curvature motion, where $$F(p,X)=-\tr\left( \left( I-\frac{p\otimes p}{|p|^2}  \right)X\right).$$} (see Definition \ref{enveloppes}). \\
To play the role of the obstacles, we consider $u^-$ and $u^+ : \Rn \times [0,+\infty) \to \R$, with $u^-\ls u^+$. The function $u$ will be forced to stay between $u^-$ and $u^+$. Geometrically, the constraint reads $$\{u^+ <s\} \subset \{u <s\} \subset \{u^- <s\},$$
where, given two functions $u,v : \R^n \times [0,\infty) \mapsto \R$, we will denote by 
$$\{u = v\} := \{(x,t) \in \R^n \times [0,\infty) \ \vert \ u(x,t) = v(x,t) \} \quad \text{ and}$$
$$\{u < v \} := \{(x,t) \in \R^n \times [0,\infty) \ \vert \ u(x,t) < v(x,t) \}.$$

To adapt the classical theory of viscosity solutions (we will use the same scheme of proof as in \cite{c92user}), the key point is to define correctly sub and super solutions of
\begin{equation}
 u_t + F(Du, D^2u) + k |Du| = 0 \quad \text{with} \quad u^- \ls u \ls u^+.
\label{pb}
\end{equation}
This definition for two obstacles has been already given, for instance in \cite{yam87}. To state it, we fisrt need the following notation.
\begin{defi}
 For $f : \Rn \to \R$, we denote by $f^\ast$ the upper semicontinuous envelope of $f$. More precisely
$$f^\ast(x) = \limsup_{y\to x} f(y).$$
We define in a similar way the lower semicontinuous envelope of $f$.
$$f_\ast (x) = \liminf_{y \to x} f(y).$$
Note that $f^\ast$ (resp. $f_\ast$) is the smallest (resp. largest) semicontinuous function $g$ such that $g \gs f$ (resp. $g \ls f$).
\label{enveloppes} 
\end{defi}

We are now ready to give the main definition.
\begin{defi}
\label{defsol}
 A function $u : \Rn \times \R^+ \to \R$ is said to be a (viscosity) subsolution on $[0,T)$ of the motion equation with obstacles $u^+,u^-$ and initial condition $u_0$ if 
\begin{itemize}
 \item $u$ is upper semicontinous (usc),
 \item for all $(x,t) \in \Rn \times [0,T), \, u^-(x,t) \ls u(x,t) \ls u^+(x,t)$,
\item for all $x \in \Rn$, $u(x,0)\ls u_0(x)$,
 \item if $\varphi$ is a $\mathcal C^2$ function of $x,t$, if $(\hat x,\hat t) \in \Rn \times (0,T)$ is a maximizer of $u-\varphi$ and if $u(\hat x,\hat t) > u^-(\hat x, \hat t)$, then, at $(\hat x, \hat t)$,
$$\varphi_t + F_\ast (D\varphi,D^2 \varphi) + k|D\varphi| \ls 0.$$ 
\end{itemize}
Similarly, $u$ is said to be a (viscosity) supersolution of the motion equation with obstacles $u^+,u^-$ and initial condition $u_0$ if 
\begin{itemize}
 \item $u$ is lower semicontinous (lsc),
 \item for all $(x,t) \in \Rn \times [0,T), \, u^-(x,t) \ls u(x,t) \ls u^+(x,t)$,
\item for all $x \in \Rn$, $u(x,0)\gs u_0(x)$,
 \item if $\varphi$ is a $\mathcal C^2$ function of $x,t$, if $(\hat x,\hat t) \in \Rn \times (0,T)$ is a minimizer of $u-\varphi$ and if $u(\hat x,\hat t) < u^+(\hat x, \hat t)$, then at $(\hat x, \hat t)$,
$$\varphi_t + F^\ast (D\varphi,D^2 \varphi)+ k|D\varphi| \gs 0.$$
\end{itemize}
Finally, $u$ is said to be a (viscosity) solution of the motion equation with obstacles $u^+,u^-$ if $u$ is both a super and a sub solution. \\
To simplify, we write
\begin{equation} u_t+F(Du,D^2u)+ k|Du| = 0 \quad \text{on }\{u^- \ls u \ls u^+\} \label{vsol}.
\end{equation}
A supersolution (resp subsolution) of the motion equation with obstacles $u^+,u^-$ will be called a supersolution (resp. subsolution) of \eqref{vsol}.
\end{defi}
Looking at the very definition, one can make the
\begin{rem}
\label{remdef}
Let $u$ be a subsolution with obstacles $u^- \ls u^+$. Then, $u$ is a subsolution with obstacles $u^-$ and $v^+$ for every $v^+ \gs u^+.$ \\
The obstacle $u^-$ is a subsolution whereas $u^+$ is a supersolution.
\end{rem}

\begin{rem}
It has to be noticed that using this definition, obstacles can depend on the time variable. Moreover, the contact zone $\{u^+ = u^-\}$ can be nonempty.

We also want to point out that the obstacle problem can be defined using a modified $F$ (see \cite{c92user}, Example 1.7). For instance, let
\begin{equation} G(x,t,u,u_t,Du,D^2u) = \max \left( \min\left( u_t + F(Du, D^2u) , u - u^- \right), u-u^+ \right). \label{minmax} \end{equation}
One can easily show that the (usual) viscosity solutions of $G = 0$ coincide with our definition above (the only difference is the subsolutions of $G=0$ do not have to satisfy $u\gs u^-$, but must remain below $u^+$).
Nonetheless \eqref{minmax} cannot be written of the form
$$ u_t + \tilde G(x,t,u,Du,D^2 u ) = 0,$$
which is the usual form for parabolic equations, for which known results (see \cite{c92user,giga90,chen91}) could apply. Thus, despite of this convenient formulation, we have to check that the usual results still apply. That is why we decided to use the definition above with a standard function $F$ but with (explicit) obstacles.
\end{rem}

There is another equivalent definition of such solutions, which can be useful (see \cite{c92user}). 
\begin{defi}
Let $f : \Rn \times (0,T) \to \R$. We said that $(a,p,X) \in \R \times \Rn \times \mathcal S_n(\R)$ is a superjet for $f$ at $(x_0,t_0)$ and we denote $(a,p,X) \in \mathcal P^{2,+} f(x_0,t_0)$ if, for $(x,t) \to (x_0,t_0)$ in $\Rn \times (0,T)$,
$$ f(x,t)\ls f(x_0,t_0)+a(t-t_0)+\scal{p}{x-x_0} + \frac 12 \scal{X(x-x_0)}{x-x_0} + o(|t-t_0| + |x-x_0|^2).$$
We likewise say that $(a,p,X) \in \R \times \Rn \times \mathcal S_n(\R)$ is a subjet for $f$ at $(x_0,t_0)$ and we denote $(a,p,X) \in \mathcal P^{2,-} f(x_0,t_0)$ if, for $(x,t) \to (x_0,t_0),$
$$ f(x,t)\gs f(x_0,t_0)+a(t-t_0)+\scal{p}{x-x_0} + \frac 12 \scal{X(x-x_0)}{x-x_0} + o(|t-t_0| + |x-x_0|^2).$$
Then, $u$ is a subsolution of \eqref{vsol} if it satisfies the three first assumptions of the previous definition and if
$$\forall (x,t) \in \Rn \times (0,T), \ (a,p,X) \in \mathcal P^{2,+} u(x,t), \quad u(x,t) > u^-(x,t) \Rightarrow a+F_\ast (p,X) + k|p| \ls 0.$$
Of course, $u$ is a supersolution of \eqref{vsol} if the three assumptions of the first definition are satisfied and if
$$\forall (x,t) \in \Rn \times (0,T), \ (a,p,X) \in \mathcal P^{2,-} u(x,t), \quad u(x,t) < u^+(x,t) \Rightarrow a+F^\ast (p,X) + k|p|  \gs 0.$$
\end{defi}

\section{Existence and uniqueness}
The aim of this section is to show the
\begin{thm}
We assume that $u^-$ and $u^+$ are uniformly continuous and bounded and that $k$ is Lipschitz. Then, if $u_0 : \Rn \to \R$ is uniformly continuous and $u^-(x,0) \ls u_0(x) \ls u^+(x,0)$, \eqref{vsol} has an unique solution, which is uniformly continuous.
\end{thm}
The structure of the proof is classical when dealing with viscosity solutions. A comparison principle will show uniqueness, and existence will follow by standard methods. \\
In what follows, $L$ is a Lipschitz constant of $k$ and $\omega$ is a modulus of continuity for $u_0$, $u^-$ and $u^+$.

\subsection{Uniqueness}
We begin by proving a comparison principle, adapted from \cite{c92user}, Theorem 8.2. It has to be noticed that the same result with no obstacles has been proved in \cite{giga90} (Th. 4.1) in a very general framework. We could adapt this result to the obstacle case but we prefer to present a simpler and self consistent proof based on \cite{c92user} (nonetheless, we will use some ideas of \cite{giga90}).
 
\begin{prop}[Comparison principle]
We assume that $u$ is a subsolution and $v$ a supersolution of $\eqref{vsol}$ on $(0,T)$, and that $u(x,0)\ls v(x,0)$. Then, $u \ls v$ in $\Rn \times (0,T)$.
\label{compple}
\end{prop}
\begin{proof}
We proceed by contradiction. Since for every $c$ sufficiently small, we can find $\eta > 0$ such that $\tilde u = (u-\frac{\eta}{T-t}) \vee u^-$ is still a subsolution, but with 
$$F(D\tilde u, D^2\tilde u) + k|D\tilde u| \ls -c < 0,$$ it is enough to prove the comparison principle with $\tilde u$ and then pass to the limit (nonetheless, we still write $u$). Suppose that there exists $\overline x,\overline t$ such that $u(\overline x,\overline t)-v(\overline x,\overline t) \gs 2 \delta >0.$ One defines
$$\Phi(x,y,t)=u(x,t)-v(y,t)-\frac{\alpha}{4} |x-y|^4 - \frac{\varepsilon}{2}(|x|^2 + |y|^2).$$
If $\varepsilon$ is sufficiently small, $\Phi(\overline x,\overline x,\overline t) \gs \delta.$ Hence, $M:=\max \limits_{x,y,t} \Phi(x,y,t) \gs \delta$ (the penalization at infinity $\frac 12 \eps(|x|^2 + |y|^2)$ reduces searching for the maximum to a compact set). Let $\hat x,\hat y, \hat t$ be a maximum point. Since $u$ and $v$ are bounded, there is $C$ depending only on $\Vert u \Vert_\infty$ and $\Vert v \Vert_\infty$ such that 
$$|\hat x - \hat y| \ls \frac{C}{\alpha^{1/4}}.$$

First, let us show by contradiction that $u(\hat x,\hat t) > u^-(\hat x, \hat t)$ and $v(\hat y,\hat t) < u^+(\hat y,\hat t)$. Suppose for example that $u(\hat x,\hat t) = u^-(\hat x, \hat t)$. Then
\begin{align*}0<\delta&\ls u^-(\hat x,\hat t)-v(\hat y,\hat t) \ls u^-(\hat y,\hat t) + \omega(|\hat x-\hat y|) -v(\hat y,\hat t) \\ &\ls \omega(|\hat x- \hat y|) + 0 \ls \omega(C\alpha^{-1/4}).
\end{align*}
Hence, if $\alpha$ is sufficiently large (independently of $\varepsilon$), $\omega(C \alpha^{-1/4}) \ls \delta /3$. Contradiction (this shows moreover that $\hat t < T$). Similarly, $v(\hat y,\hat t) < u^+(\hat y, \hat t)$.

In what follows, $\alpha$ is fixed sufficiently big  to satisfy these conclusions.

As \begin{equation}M + \alpha |x-y|^4 + \frac{\varepsilon}{2}(|x|^2+|y|^2) \gs u( x, t)-v( y,t) \label{lemIshii}\end{equation} with equality in $\hat x,\hat y, \hat t$, we are able to apply Ishii's lemma \cite{c92user} to 
$$u(x,t) - v(y,t) - \Phi(x,y,t) \quad \text{ where } \quad \Phi(x,y,t) = M + \alpha |x-y|^4 + \frac{\varepsilon}{2}(|x|^2+|y|^2)$$
which provides, for every $\mu >0$, $(a,b,X,Y)$ such that $(a,\underbrace{\alpha|\hat x - \hat y|^2(\hat x - \hat y)}_{=:\hat p}-\varepsilon \hat y,Y) \in \overline{\mathcal P}^{2,+}v(\hat y, \hat t)$ and $(b,\alpha|\hat x - \hat y|^2(\hat x - \hat y)+\varepsilon \hat x,X) \in \overline{\mathcal P}^{2,-}u(\hat x, \hat t)$. It provides moreover $a-b=0$ and 
$$-\left( \frac{1}{\mu}+ \Vert A \Vert \right) \begin{bmatrix} I & 0 \\ 0 & I \end{bmatrix} \ls \begin{bmatrix} X & 0 \\ 0 & -Y \end{bmatrix} \ls A + \mu A^2,$$
where 
$$A = D^2 \Phi (\hat x, \hat y,\hat t) = \begin{bmatrix} P & -P \\ -P & P \end{bmatrix} + \varepsilon \begin{bmatrix} I & 0 \\ 0 & I \end{bmatrix} $$
and
$$ P  = 2 \alpha (\hat x - \hat y) \otimes (\hat x - \hat y) + \alpha |\hat x - \hat y |^2.$$

That shows in particular that $X-Y \ls \eps^2 I$ and $\Vert X\Vert, \Vert Y \Vert \ls C_1 \left(\alpha |\hat x - \hat y|^2 + \varepsilon \right).$ 

Since $u$ and $v$ are respectively subsolution and supersolution near $(\hat x ,\hat t)$ and $(\hat y, \hat t)$, one has
$$
c \ls  a-b + F^\ast (\hat p-\varepsilon \hat y, Y-\varepsilon I)-F_\ast (\hat p+\varepsilon \hat x, X+\varepsilon I) + k(\hat y,\hat t)|\hat p-\varepsilon y|- k(\hat x, \hat t)|\hat p +\varepsilon \hat x |.
$$
One can write
$$k(\hat y, \hat t)|\hat p +\varepsilon \hat y | - k(\hat x,\hat t)|\hat p-\varepsilon x| \\ \ls (k(\hat y,\hat t) - k(\hat x,\hat t))|\hat p +\varepsilon \hat y | + 2|k(\hat x,\hat t)|(|\varepsilon \hat y| + |\varepsilon \hat x|),$$
which gives, with $a-b=0$,
$$
c \ls F^\ast (\hat p-\varepsilon y , Y-\varepsilon I)  - F_\ast (\hat p +\varepsilon \hat x,X+\varepsilon I)\\+ L(|\hat x - \hat y|)|\hat p +\varepsilon \hat y | + 2\Vert k \Vert_{\infty}(|\varepsilon \hat x| + |\varepsilon \hat y|).
$$
Then, we want to let $\varepsilon$ go to 0. \\
Since $M \gs \delta >0$, we have
$$\delta + \frac{1}{4}\alpha |\hat x-\hat y|^4 + \frac{\varepsilon}{2}(|\hat x|^2+|\hat y|^2) \ls u(\hat x, \hat t)-v(\hat y,\hat t) \ls \Vert u \Vert_\infty + \Vert v \Vert_\infty,$$
which implies that $\varepsilon |\hat x|^2$ is bounded, hence $\varepsilon \hat x \to 0$ (same for $\varepsilon \hat y$), whereas for $i \in \{2,3,4\},$ $\alpha |\hat x - \hat y|^i$ is bounded (so is $\hat p$, $X$ and $Y$). Indeed, $\alpha$ is fixed here. Hence one can assume that $\hat p \to p$, $X \to X_0$, $\alpha |\hat x - \hat y|^4 \to \mu_\alpha.$ \\
We now use a short lemma, which is an easy adaptation of \cite{giga90}, Proposition 4.4 (see also Lemma 2.8 in the preprint of \cite{for08}, which has a form which is closer to ours) and whose proof is reproduced here for convenience.
\begin{lem}
One has 
$$\lim_{\alpha \to \infty} \lim_{\varepsilon \to 0} \alpha |\hat x - \hat y|^4 =0.$$
\end{lem}
\begin{proof}
Let $$M_h=\sup \limits_{\substack{|x-y| \ls h\\t \in [0,T)}} (u(x,t)-v(y,t))$$ and $(x_h^n,y_h^n,t_h^n)$ such that $u(x_h^n,t_h^n) - v(y_h^n,t_h^n) \gs M_h - \frac 1n$ and $|x_h^n-y_h^n|\ls h.$ Then,
$$M_h - \frac 1n - \frac{\alpha h^4}{4} - \frac \varepsilon2\left(|x_h^n|^2 + |y_h^n|^2\right) \ls M \ls  u(\hat x, \hat t) - v(\hat y,\hat t).$$
As $x_h^n$ and $y_h^n$ do not depend on $\varepsilon$, one can let it go to zero (considering the liminf of the right term) to get
$$M_h-\frac 1n - \frac{\alpha h^4}{4} \ls \liminf_{\varepsilon \to 0} (u(\hat x,\hat t)-v(\hat y,\hat t)).$$
Let $h\to 0$ (We denote by $M'$ the decreasing limit of $M_h$). One obtains
$$M' - \frac 1n \ls \liminf_{\varepsilon \to 0} (u(\hat x,\hat t) - v(\hat y,\hat t)).$$
Let $\alpha$ go to infinity:
\begin{align*}
 M'-\frac 1n &\ls \liminf_{\alpha \to \infty} \liminf_{\varepsilon \to 0} ( u(\hat x,\hat t) - v(\hat y,\hat t)) \\
&\ls \limsup_{\alpha \to \infty} \left( \sup_{\substack{|x-y|\ls C\alpha^{-1/4} \\ t\in [0,T)}} (u(x,t)-v(y,t))\right) \\
&\ls \limsup_{h\to 0} \sup_{|x-y|\ls h} (u(x,t)-v(y,t)) = M'
\end{align*}
hence
$$\lim_{\alpha \to \infty} \lim_{\varepsilon \to 0} u(\hat x,\hat t) - v(\hat y, \hat t) = M'.$$
We prove similarly that $\lim \limits_{\alpha \to \infty} \lim \limits_{\varepsilon \to 0} M = M'$.
As a matter of fact,
$$\lim_{\alpha \to \infty} \lim_{\varepsilon \to 0} \left( \alpha |\hat x - \hat y|^4 + \frac{\varepsilon}{2} (|\hat x|^2 + |\hat y|^2) \right) = 0,$$
which proves the lemma.
\end{proof}
One can now choose $\alpha$ such that $\lim \limits_{\varepsilon \to 0} \alpha |\hat x - \hat y|^4 \to \mu_\alpha$ with $\mu_\alpha \ls c /2L$ and pass to the liminf in $\varepsilon \to 0$. One gets (using $X \ls Y + \varepsilon^2 I$),
$$\frac{c}{2} \ls \liminf \left( F_\ast (\hat p ,X) - F^\ast (\hat p , X) \right).$$
To conclude, we distinguish two cases:
\begin{itemize}
 \item if $p \neq 0$, then $F^\ast (p,X_0) = F_\ast (p,X_0)$ and we get the contradiction.
\item if $p =0$, we have $\alpha|\hat x - \hat y|^2(\hat x - \hat y) \underset{\varepsilon \to 0}\longrightarrow 0$, so $X_0=0$ and $F^\ast (p,X_0)=F_\ast (p,X_0)=0$ and we get the contradiction too.
\end{itemize}

\end{proof}

\subsection{Existence}
We will build a solution using Perron's method. Since we know that the supersolutions of \eqref{vsol} remain larger than subsolutions, the solution, if it exists, must be the largest subsolution (or equivalently, the smallest supersolution). Hence we introduce
$$W(x,t)=\sup\{w(x,t),\; \text{$w$ subsolution on $[0,T)$}\}.$$
We show that $W$ is in fact the expected solution to \eqref{vsol}.

Let us first state a straightforward but useful proposition.

\begin{prop}

\begin{enumerate}[i)] 
\item Let $u$ be a subsolution of the motion without obstacles which satisfies $u \ls u^+$. Then, $u_{ob}:=u\vee u^-$ is a subsolution of \eqref{vsol} with obstacles (the same happens for $v \gs u^-$ supersolution and $v_{ob} = v \wedge u^+$).
\item More generally, if $u$ is a solution of the motion with initial conditions $u_0$ and obstacles $(u^-,u^+)$ and if $v^-$ and $v^+$ are other obstacles which satisfy $u^-\ls v^- \ls u^+ \ls v^+$, then $u\vee v^-$ is a subsolution of the equation with initial condition $u_0 \vee v^- \vert_{t=0}$ and obstacles $v^-$ and $u^+$. In addition, $u$ is a subsolution of the equation with initial conditions $u_0$ and obstacles $u^-$, $v^+$.
\end{enumerate}
 \label{wedgevee}
\end{prop}
\begin{proof}
 The proof is quite simple: consider a smooth function $\varphi$ and some $x_0,t_0$ such that $\varphi - u \vee u^-$ has a maximum at $(x_0,t_0)$. Then, using the definition of subsolutions, either $u(x_0,t_0)\vee u^-(x_0,t_0) = u^-(x_0,t_0)$ and nothing has to be done, or $u(x_0,t_0) > u^-(x_0,t_0)$. In the second alternative $(x_0,t_0)$ is in fact a maximum of $u-\varphi$. Since $u$ is a viscosity subsolution of the motion, we have $\varphi_t + F_\ast (D\varphi,D^2\varphi) + k|D\varphi| \ls 0$, what was expected.

Let us now show the second part of the proposition. The initial condition $u \vee v^- \ls u_0 \vee v^- \vert_{t=0}$ is satisfied. Once again, we consider $\varphi$ smooth and $(x_0,t_0)$ such that $u \vee v^- - \varphi$ has a maximum at $(x_0,t_0)$. Then, either $u(x_0,t_0) \vee v^-(x_0,t_0) = v^-(x_0,t_0)$ and nothing has to be checked, or $u(x_0,t_0) > v^-(x_0,t_0)$. The latter implies that $u(x_0,t_0) > u^-(x_0,t_0)$, so $\varphi_t + F_\ast (D\varphi,D^2 \varphi)+ k|D\varphi|\ls 0$, what was wanted.
\end{proof}

\begin{lem}
 Let $\mathcal F$ be a  family of subsolutions of \eqref{vsol} and define $U(x,t) :=\sup\{u(x,t)),u\in \mathcal F\}$. Then, $U^\ast$ is a subsolution of \eqref{vsol}. \label{supsol}
\end{lem}
To prove this lemma, we need the following proposition which will be useful later.
\begin{prop}
 Let $v$ be a upper semicontinuous function, $(x,t) \in \Rn \times \R$ and $(a,p,X) \in \mathcal P^{2,+} v(x,t)$. Assume there exists a sequence $(v_n)$ of usc functions which satisfy
\begin{enumerate}[i)]
 \item There exists $(x_n,t_n)$ such that $(x_n,t_n,v_n(x_n,t_n)) \to (x,t,v(x,t))$
\item $(z_n,s_n) \to (z,s)$ in $\Rn \times \R$ implies $\limsup v_n(z_n,s_n) \ls v(z,s).$
\end{enumerate}
Then, there exists $(\hat x_n,\hat t_n)\in \Rn\times \R,\; (a_n,p_n,X_n) \in \mathcal P^{2,+}v_n(\hat x_n,\hat t_n)$ such that
$$(\hat x_n ,\hat t_n, v_n(\hat x_n, \hat t_n ),a_n,p_n,X_n) \to (x,t,v(x,t),a,p,X).$$
\label{limjet}
\end{prop}
The proof of the proposition and the lemma can be found in \cite{c92user}, Lemma 4.2 and Proposition 4.3 (with obvious changes due to the parabolic situation and obstacles).

In our way to prove that $W$ is the solution of \eqref{vsol}, we need to show that it is a subsolution of \eqref{vsol}. Lemma \ref{supsol} shows that $W^\ast$ is a subsolution of \eqref{vsol} with obstacles, but without taking the initial condition into account. Indeed even if for all subsolution, one has $u(x,0) \ls u_0(x)$, which implies $W(x,0) \ls u_0(x)$, taking the semicontinuous envelope could break this inequality. We thus need to build some continuous barriers which will force $W^\ast$ to remain below $u_0$ at time zero. More precisely, we build a continuous supersolution $w^+$ which gets the initial data $u_0$. Then, by comparison principle, every subsolution $u$ will satisfy $u \ls w^+$ and $W \ls w^+$. Taking the envelope will yield
$$W^\ast \ls (w^+)^\ast = w^+$$ which will imply
$$W^\ast(x,0) \ls u_0(x).$$

Similarly, we build a continuous subsolution $w^-$ which also gets the initial data. By the very definition of $W$, it gives $W(x,0) \gs u_0(x).$

For technical reasons, we begin building the solution in the case where $k=0$.
\subsubsection{Construction of barriers in the non forcing case}
\label{barriernf}

Let us construct $w^-$. Without a forcing term, we note that for all $\xi \in \R^n$ and $A,B$ with $B$ sufficiently large relatively to $A$,
$$\tilde h(x,t)=-(A |x-\xi|^2+ B t)$$ is a subsolution of \eqref{vsol} in a neighborhood of $\xi$ but with neither initial conditions nor obstacles. 
We define
$$h(x,t) =  h(x,t) \vee u^-(x,t).$$
Then, $h$ is a subsolution (on the full domain, since as soon $|x-\xi| \gs \Vert u^- \Vert_\infty/A$, $h(x,t) = u^-(x,t)$) of \eqref{vsol}, for $A$ sufficiently large uniformly in $\xi$. We then define
$$\theta_\xi(r)= \inf\{u_0(y) \, \mid \, A |y-\xi|^2 + r \ls 0\}$$
The function $\theta_\xi$ is bounded, non decreasing, continuous and satisfies $\theta_\xi(0)=u_0(\xi)$ and $\theta_\xi (-A |x-\xi|^2-Bt) \ls u_0(x).$ As the equation is geometric, $\theta_\xi(-A|x-\xi|^2 - Bt) \vee u^-(x,t)$ is also a subsolution. Let us then define
$$\phi(x,t)= \left( \sup_\xi \theta_\xi(-A|x-\xi|^2 - Bt)  \vee u^-(x,t) \right)^\ast.$$
Since $\theta_\xi (-A|x-\xi|^2-Bt) \ls u_0(x)$ and $u_0$ is continuous, we also have $\phi(x,t) \ls u_0$. In addition, we can check that 
\begin{equation}\phi(x,t) \gs \theta_x(-A|x-x|^2 -Bt) = \theta_x(-Bt)\gs u_0(x) - \omega(\sqrt{\frac{Bt}{A}}).\label{controlinit}\end{equation} Hence, $\phi(x,0) = u_0(x)$.
Thanks to Lemma \ref{supsol}, $\phi$ is a subsolution with $\phi(x,0)\ls u_0(x).$
We conclude this proof defining
$$w^-(x,t)= \left(\phi(x,t)-\omega(t)\right) \vee u^-(x,t).$$
It is clear that $w^-$ is a subsolution with obstacles. Indeed, by definition, $w^- \gs u^-$. Moreover, $\phi(x,t) - \omega(t) \ls u_0(x) - \omega(t) \ls u^+(x,0) - \omega(t) \ls u^+(x,t).$ Proposition \ref{wedgevee} concludes the proof.

The other barrier $w^+$ is obtained similarly: 
$$w^+ = \left( \inf_\xi \theta^\xi(A|x-\xi|^2 + Bt)  \wedge u^+(x,t) \right)_\ast \wedge u^+(x,t)$$
with
$$\theta^\xi(r) = \sup \{u_0(y) \, \mid \, A|y-\xi|^2 - r \ls 0 \}.$$

\subsubsection{Perron's method}
We have just seen that, thanks to the barriers, $W^\ast$ is a subsolution of \eqref{vsol}. We now want to show that $W$ is actually a subsolution and that it is also a supersolution.

First, we show uniform continuity of the function $W$, which shows that $W^\ast = W$ and therefore, that

\begin{rem}
If $k(x,t)=0$, then $W$ is $\omega$-uniformly continuous in space. In time, $W$ is uniformly continuous with modulus $\tilde \omega : r \mapsto \max(\omega(r), \omega(\sqrt{\frac{Br}{A}}))$, where $B$ is the constant introduced when constructing the barriers. Indeed, the proof is contained in the following lemma.
\begin{lem}
Let $u(x,t)$ be a subsolution of \eqref{vsol} with no forcing term (and $u_0,u^-,u^+$ $\omega$-uniformly continuous in space and time). Then, for $t>0$ and $z\in \R^n$,
$$u_{z,\delta}(x,t)=(u(x+z,t+\delta)-\omega(|z|) -\tilde \omega(|\delta|)) \vee u^-(x,t)$$ is also a subsolution.
\end{lem}
\begin{proof}
To begin, we notice that $u(x+z,t+\delta)-\omega(|z|)-\tilde \omega(|\delta|)\ls u^+(x,t)$. \\
Now, let $\varphi$ be a smooth function with $\forall x,t$, $u_{z,\delta}(x,t)\ls \varphi(x,t)$ with equality at $(\overline x,\overline t)$. Then, either $u_{z,\delta}(\overline x,\overline t) = u^-(\overline x,\overline t)$, and nothing has to be done, or $u_{z,\delta}(\overline x,\overline t) > u^-(\overline x,\overline t)$. In the second alternative, we have 
$$u(\overline x+z,\overline t + \delta)-\omega(|z|) -\tilde\omega(\delta) > u^-(\overline x,\overline t) = u^-(\overline x+z,\overline t + \delta)+(u^-(\overline x,\overline t)-u^-(\overline x+z,\overline t+\delta))$$
hence
$$u(\overline x+z,\overline t+\delta) > u^-(\overline x+z,\overline t+\delta)+\underbrace{(u^-(\overline x,\overline t)-u^-(\overline x+z,\overline t+\delta) + \omega(|z|) + \tilde\omega(|\delta|))}_{\gs 0} \gs u^-(\overline x+z,\overline t+\delta).$$
As $u$ is a subsolution at $(\overline x+z,\overline t+\delta)$ and $u(x+z,t+\delta)\ls \varphi(x,t)+\omega(|z|) + \tilde\omega(|\delta|)$ with equality at $(\overline x+z,\overline t+\delta)$, one can write, with $y=x+z, s=t+\delta$,
$$u(y,s)\ls \varphi(y-z,s-\delta)+\omega(|z|) + \tilde\omega(|\delta|) =:\phi(y,s),$$ equality at $(\overline y,\overline s)$ (with $\overline{y} := \overline x +z$ and $\overline s = \overline t + \delta$), and deduce that $\phi_t+F_\ast (D\phi(\overline y,\overline s),D^2\phi(\overline y,\overline s))\ls 0.$ Since $D\phi(\overline y,\overline s)=D\varphi(\overline x, \overline t)$ (so are the time and spatial second derivatives), we get $$\varphi_t+F_\ast (D\varphi(\overline x,\overline t),D^2\varphi(\overline x,\overline t))\ls 0,$$ what was expected.

Concerning the initial conditions, we have (we use \eqref{controlinit} and the comparison principle Proposition \ref{compple} between $u$ and $w^+$)
$$u(x+z,0+\delta) - \omega(|z|) - \tilde\omega(\delta) \ls w^+(x+z,\delta) - \omega(|z|) - \tilde\omega(|\delta|) \ls u_0(x+z) - \omega (|z|) \ls u_0(x).$$
\end{proof}
Applying this lemma to $W$ shows $(x,t)\mapsto W(x+z,t+\delta) - \omega(|z|) - \tilde\omega(|\delta|) \vee u^-(x+z,t)$ is a subsolution. By definition of $W$, one can write
$$W(x,t) \gs (W(x+z,t+\delta) - \omega(|z|) -\tilde \omega(\delta)) \vee u^-(x+z,t) \gs W(x+z,t+\delta)-\omega(|z|) - \tilde\omega(\delta)$$
which shows exactly that $W$ is uniformly continuous.
\label{noforcing}
\end{rem}

We now want to show that $W$ is in fact a supersolution of \eqref{vsol}. We need the following lemma which is adapted from \cite{c92user}, Lemma 4.4.
\begin{lem}
Let $u$ be a subsolution of \eqref{vsol}. If $u_\ast$ fails to be a solution of $u_t+F^\ast (Du,D^2u)+k|Du| \gs 0$ at some $(\hat x,\hat t)$ (there exists $(a,p,X) \in \mathcal P^{2,-} u_\ast (\hat x,\hat t)$ such that $a+F^\ast (p,X) + k|p| <0$), then for all sufficiently small $\kappa$, there exists a solution $u_\kappa$ of $u_t+F_\ast (Du,D^2u)+k|Du| \ls 0$ satisfying $u_\kappa (x,t) \gs u(x,t)$, $\sup \limits_{\Rn} (u_\kappa-u) >0$, $u_\kappa (x,t) \ls u^+(x,t)$ and such that $u$ and $u_\kappa$ coincide for all $|x-\hat x|,|t-\hat t| \gs \kappa.$ \label{notsol}
\end{lem}
\begin{proof}
We can suppose that $u_\ast$ fails to be a supersolution at $(0,1)$ (this implies in particular $u_\ast(0,1)<u^+(0,1)$). We get $(a,p,X) \in \mathcal P^{2,-} u_\ast (0,1)$ such that $a+F^\ast (p,X) +k(0,1)|p| <0$. We introduce for $\gamma,\delta,r >0$,
$$u_{\delta,\gamma} (x,t)=u_\ast(0,1)+\delta +\scal{p}{x} + a(t-1) +\frac 12 \scal{Xx}{x} - \gamma (|x|^2+t-1).$$
By upper semicontinuity of $F^\ast $, $u_{\delta,\gamma}$ is a subsolution of $u_t+F^\ast (Du,D^2u)+k|Du| \ls 0$ on $B_r((0,1))$ for $\gamma, \delta,r$ sufficiently small. \\
Since
$$u(x,t) \gs u_\ast(x,t) \gs u_\ast(0,1)+a(t-1)+\scal{p}{x} + \frac 12 \scal{Xx}{x} +o(|x|^2)+o(|t-1|),$$
choosing $\delta = \gamma \frac{r^2+r}{8}$, we get $u(x,t) > u_{\delta,\gamma}(x,t)$ for $\frac{r}2 \ls |x|,|t-1|\ls r$ and $r$ sufficiently small. Moreover, we can reduce $r$ again to have $u_{\delta,\gamma} \ls u^+$ on $B_r$ (Choosing $r$ sufficiently small, one has $\delta$ sufficiently small and $u_{\delta,\gamma} (0,1)-u_\ast(0,1) = \delta < u^+(0,1)-u_\ast(0,1)$. By continuity, one can find a smaller $r$ such that  $u_{\delta,\gamma} (x,t)< u^+(x,t)$ for all $\frac r2 \ls |x|, |t-1| \ls r$.).  \\
Thanks to Lemma \ref{supsol}, the function
$$\tilde u(x,t) = \left\{ \begin{matrix} \max(u(x,t),u_{\delta,\gamma}(x,t)) \text{ if }|x,t-1|<r \\ u(x,t) \text{ otherwise} \end{matrix} \right.$$
is a subsolution of \eqref{vsol} (with initial conditions if $r$ is small enough).
\end{proof}

Now, we saw that $W$ is a subsolution of \eqref{vsol} (in particular, $W \ls u^+$). If it is not a supersolution at a point $\hat x,\hat t$ , Lemma \ref{notsol} provides $W_\kappa \gs W$ subsolutions of \eqref{vsol} (with initial condition, even if we have to reduce $r$ again, to make $t$ stay far from zero), which is a contradiction with the definition of $W$. \\
Finally, $W$ is the expected solution of \eqref{vsol}.

\subsubsection{With forcing term}
\begin{enumerate}
\item We assume at this point only that $u^-,u^+$ and $u_0$ are $K$-Lipschitz in space. Then, thanks to Remark \ref{noforcing}, there exists a $K$-Lipschitz (in space) solution $\psi$ of the non forcing term equation. Let us set $w^-(x,t)=(\psi(x,t)+\Vert k \Vert_\infty K t) \vee u^-(x,t).$ It satisfies, as soon as $w^->u^-$,
$$u_t-\Vert k \Vert_\infty K +F(Du,D^2u) = 0, \qquad u(x,0)=u_0(x).$$
As a consequence, $w^-$ is a continuous subsolution of \eqref{vsol} (with forcing term) satisfying $w^-(x,0)=u_0(x)$. It is a barrier as in \ref{barriernf}. We build $w^+$ in a similar way and apply Perron's method to see that $W$ is a solution.
\item Here, $u^+$, $u^-$ and $u_0$ are only $\omega$-uniformly continuous. For all $K >0$, let $u^0_K = \min\limits_{y} u_0(y) + K|x-y|$, $u^+_K(x,t)=\max\limits_{y} u^+(y,t) - K|x-y|$ and $u^-_K=\min\limits_{y} u_0(y)+K|x-y|.$ These three new function are $K$-Lipschitz in space and converge uniformly (in space) to $u_0,u^+$ and $u^-$ when $K \to \infty.$ Moreover, as $u_0,u^+,u^-$ are $\omega$-uniformly continuous, so are they. \\
Thanks to the previous point, for every $K$, there exists a solution $u_K$ of \eqref{vsol} with obstacles $u^+_K,u^-_K$ and with initial data $u^0_K$, which is (thanks to the following proposition \ref{regularsol}, which is admitted for a little time) uniformly continuous with same moduli on $[0,T]$ for every $T$. One can define, thanks to Ascoli's theorem
$$u(x,t)=\lim_n u_{K_n}(x,t).$$
The function $u$ is continuous. We have to check that it is the solution of the motion with obstacles $u^\pm$. 

It is clear that $u^- \ls u \ls u^+$. Let $\varphi$ be a smooth function and $(\hat x,\hat t)$ a maximum point of $u-\varphi$ such that $u(\hat x, \hat t) - u^-(\hat x,\hat t) =: \eta >0$. One can assume that the maximum is strict. We then choose $\varepsilon$ such that 
$$\forall (x,t) \in B_\varepsilon(\hat x,\hat t), \quad u(x,t)-u^-(x,t) \gs \frac {3\eta} 4.$$
Let $$\delta := \min_{\partial B_\varepsilon} |u - \varphi|.$$
It is positive (since the maximum is strict, possibly reducing $\eps$). We choose $n_0$ such that 
$$\forall n \gs n_0, \; \Vert u-u_{K_n} \Vert_{L^\infty(B_\varepsilon)}, \ \Vert u^- - u^-_{K_n} \Vert_{L^\infty(B_\varepsilon)} \ls \max\left( \frac{\eta}{4}, \, \frac{\delta}{2} \right).$$
Then, for every $n \gs n_0$, $u_{K_n} - \varphi$ has a maximum $(x_n,t_n)$ on $B_\varepsilon$ reached out of $u^-_{K_n}$. It is easy to show that $(x_n,t_n) \to (\hat x,\hat t)$. Since $u_{K_n}$ is a viscosity subsolution, one can write, at $(x_n,t_n)$,
$$\varphi_t + F_\ast (D\varphi, D^2\varphi) + k|D\varphi| \ls 0.$$
By smoothness of $\varphi$ and semicontinuity of $F_\ast $, we get the same inequality at $(\hat x,\hat t)$.

We prove that $u$ is a supersolution using the same arguments.
\end{enumerate}

Let us conclude this section by an estimation of the solution's regularity, which is essentially \cite{for08}, Lemma 2.15 (except that the solution here is only uniformly continuous).
\begin{prop}
Let $u$ be the unique solution of \eqref{vsol}. Then $u$ is uniformly continuous in space. moreover, one as
$$\forall (x,y,t), \quad |u(x,t)-u(y,t)| \ls \omega(e^{Lt}|x-y|).$$
 \label{regularsol}
\end{prop} 
\begin{proof} 
First, it is well known that one can choose $\omega$ to be continuous and nondecreasing. Since $u$ and $v$ are bounded, $\omega \wedge (\Vert u \Vert_\infty + \Vert v \Vert_\infty)$ is a modulus too. In the following, we use this new modulus, still denoted by $\omega$.

Then, let $\rho_n$ be a $\mathcal C^\infty$ nondecreasing function on $[0,\infty[$ such that $0\ls \rho_n-\omega$, for all $r > n+1$, $\rho_n(r)=2N+1$, and for all $r\in [0,n], \rho_n(r)-\omega(r) \ls \frac{1}{n}.$ We define
$$\omega_n(r)=\rho_n+\frac{r}{n^2}.$$
It's clear that $\omega_n(r) \underset{n\to \infty}{\longrightarrow} \omega(r).$
Moreover, for a fixed $n$, $\omega'_n(r)$ is bounded and remains far from zero. In what follows, we work with $\omega_n$.

We will proceed as in Proposition \ref{compple}. Let $\phi(x,y,t)=\omega_n(e^{Lt}|x-y|).$ We will show by contradiction that $u(x,t)-u(y,t)\ls \phi(x,y,t).$ Assume that $$M:=\sup_{(x,y,t)\in \Rn \times \Rn \times [0,T)} u(x,t)-u(y,t)-\phi(x,y,t)>0.$$
As before, we introduce $$\tilde M = \sup_{x,y,t\ls T} u(x,t)-u(y,t)-\phi(x,y,t)-\frac \alpha2 (|x|^2+|y|^2)-\frac{\gamma}{T-t}.$$ For sufficiently small $\gamma, \alpha$, $\tilde M$ remains positive and is attained (at $\overline x,\overline y,\overline t<T)$. As $u_0$ is $\omega$-uniformly continuous, $\overline t >0.$ Moreover, since $u$ is continuous, $|\overline x- \overline y|$ is bounded away from zero, independently of $\alpha$ and $\gamma$.

By assumption, $u^-(\overline x,\overline t)\ls u^-(\overline y,\overline t) + \omega(|\overline x-\overline y|) \ls u^-(\overline y,\overline t) + \omega_n(|\overline x-\overline y|)\ls u(\overline y,\overline t)+\phi(\overline x,\overline y,\overline t)$ so $0\ls \tilde M < u(\hat x,\hat t)-u(\hat y,\hat t) - \phi(\hat x, \hat y,\hat t)$ forces $u(\overline x,\overline t) > u^-(\overline x,\overline t).$ Similarly, $u(\overline y,\overline t) < u^+(\overline y,\overline t).$ 

Applying Ishii's lemma (\cite{c92user}, Th. 8.3) to $\tilde u(x,t) = u(x,t)-\frac \alpha2 |x|^2$ and $\tilde v(y,t) = u(y,t) +\frac \alpha 2  |y|^2$ where $$\overline p = D_x\phi = \frac{\overline x-\overline y}{|\overline x - \overline y|} e^{L\overline t} \omega_n'(e^{L\overline t} |\overline x - \overline y|) = -D_y\phi \neq 0,$$
\begin{multline*}Z=D^2_x\phi = \frac{e^{L\overline t}}{|\overline x-\overline y|} \omega_n'(e^{L\overline t}|\overline x - \overline y|) I + \frac{(\overline x - \overline y) \otimes (\overline x-\overline y)}{|\overline x - \overline y|^3} e^{Lt} \omega_n'(e^{L\overline t}|\overline x - \overline y|)  \\+ \frac{(\overline x - \overline y) \otimes (\overline x-\overline y)}{|\overline x - \overline y|^2} e^{2L\overline t} \omega_n''(e^{L\overline t}|\overline x - \overline y|).\end{multline*}
and
$$A=D^2 \phi =\begin{bmatrix}
     Z & -Z \\ -Z & Z
    \end{bmatrix},$$
we get the following. For all $\beta$ such that $\beta A <I$, there exists $\tau_1,\tau_2\in \mathbb R$, $X,Y \in \mathcal S_n$ such that
$$\tau_1-\tau_2= \frac{\gamma}{(T-t)^2}+Le^{L\overline t} |\overline x - \overline y|\omega_n'(e^{Lt}|\overline x-\overline y|),$$
$$(\tau_1,\overline p+\alpha \overline x,X+\alpha I) \in \overline{ \mathcal P}^{2,+} u(\overline x,\overline t),$$
$$(\tau_2,\overline p-\alpha \overline y,Y-\alpha I) \in \overline{ \mathcal P}^{2,-} u(\overline y,\overline t),$$
$$\frac{-1}{\beta} \begin{bmatrix} I & 0 \\ 0 & I \end{bmatrix} \ls \begin{bmatrix} X & 0 \\ 0 & -Y \end{bmatrix} \ls (I-\beta A)^{-1} A.$$
In particular, the last equation provides $X \ls Y$.

As $u$ is a subsolution and a supersolution, one has
\begin{equation}\tau_1+k(\overline x,\overline t)|\overline p+\alpha \overline x| + F_\ast (\overline p + \alpha \overline x, X+\alpha I) \ls 0, \label{ishiifsousuc}\end{equation}
$$\tau_2-k(\overline y,\overline t)|\overline p-\alpha \overline y| + F^\ast (\overline p - \alpha \overline y, Y-\alpha I) \gs 0.$$

$X \ls Y$ in the last equation gives
\begin{equation}
 -\tau_2+k(\overline y,\overline t)|\overline p-\alpha \overline y| - F^\ast (\overline p - \alpha \overline y, X-\alpha I) \ls 0.\label{ishiifsuruc}
\end{equation}
Adding \eqref{ishiifsuruc} to \eqref{ishiifsousuc} leads to
\begin{multline}\frac{\gamma}{(T-\overline t)^2} + Le^{L \overline t} |\overline x - \overline y|\omega_n'(e^{L\overline t}|\overline x-\overline y|) - k(\overline x,\overline t)|\overline p+\alpha \overline x| + k(\overline y, \overline t)|\overline p-\alpha \overline y| \\+ F_\ast (\overline p + \alpha \overline x, X+\alpha I) - F^\ast (\overline p - \alpha \overline y, X-\alpha I) \ls 0. \label{somme}\end{multline}

Notice that 
\begin{multline}
Le^{L\overline t} |\overline x - \overline y|\omega_n'(e^{L\overline t}|\overline x-\overline y|) -k(\overline x, \overline t)|\overline p| + k(\overline y,\overline t)|\overline p|  \\ \gs Le^{L\overline t} |\overline x - \overline y|\omega_n'(e^{L\overline t}|\overline x-\overline y|) - L|\overline x - \overline y|e^{L\overline t} \omega_n'(e^{L\overline t}|\overline x - \overline y|) \gs 0. \label{uclipsch}
\end{multline}
Then, \eqref{somme} becomes 
$$
\frac{\gamma}{(T-\overline t)^2}+\left(|\overline p| - |\overline p + \alpha \overline x|\right)k(\overline x,\overline t) - \left(|\overline p| - |\overline p - \alpha \overline y|\right) k(\overline y,\overline t) + F_\ast (\overline p + \alpha \overline x, X+\alpha I) - F^\ast (\overline p - \alpha \overline y, X-\alpha I) \ls 0.
$$
Let $\alpha$ go to zero. $\overline p$ and $X$ are bounded: one assume they converge and still denote by $\overline p,X$ their limit. As $\vert \overline p \vert \gs \frac 1{n^2}$ ($\rho_n$ is nondecrasing), $F_\ast (\overline p,H) = F^\ast (\overline p,H)$ for all $H \in \mathcal S_n$. Moreover, $\alpha \overline x, \alpha \overline y \to 0$ and $k$ is bounded, hence
$$\frac{\gamma}{(T-\overline t)^2} \ls 0,$$ which is a contradiction. So $$u(x,t)-u(y,t) \ls \omega_n(e^{Lt}|x-y|).$$ 
It remains to let $n$ go to $+\infty$ to conclude.
\end{proof}
\subsection{The motion is geometric}
In all this subsection, a solution $u$ of the motion with initial data $u_0$ and obstacles $u^-$ and $u^+$ will be denoted by $u=[u_0,u^-,u^+].$ The corresponding equation will be denoted by $(u_0,u^-,u^+)$. \\
To agree with the geometric motion, we have to check that the zero level-set of the solution depends only on the zero level sets of the initial condition $u_0$ and of the obstacles $u^+$ and $u^-$. 

\begin{lem}
Let $u=[u_0,u^-,u^+]$ and $v=[v_0,v^-,v^+]$. We assume that $u_0\ls v_0$, $u^-\ls v^-$ and $u^+ \ls v^+$. Then, $u \ls v$.
\label{compob}
\end{lem}
\begin{proof}
This proposition is obvious thanks to Remark \ref{remdef}. Indeed, $u$ is a subsolution of $(u_0,u^-,u^+)$ so is a subsolution of $(u_0,u^-,v^+)$ whereas $v$ is a supersolution of $(v_0,v^-,v^+)$, so of $(u_0,u^-,v^+).$ The comparison principle implies 
$$ u \ls v.$$
\end{proof}

\begin{prop}
Let $u$ be the solution of \eqref{motion} with obstacles $u^+$ and $u^-$, and let $\phi$ be a continuous nondecreasing function $ [-\Vert u^- \Vert,\Vert u^+ \Vert] \to \R$ such that  $\{\phi = 0\} = \{0\}$. Then, the solutions 
$$[u_0 \wedge (\phi(u^+) \vee u^-) \vert_{t=0},u^-,\phi(u^+)\vee u^-],$$ 
$$(u_0\vee (\phi(u^-) \wedge u^+)\vert_{t=0}, \phi(u^-) \vee u^+, u^+]$$ 
$$\text{and} \quad [(\phi(u_0) \wedge u^+\vert_{t=0}) \vee u^-\vert_{t=0}, u^-,u^+]$$ have the same zero level set as $u$.
\label{invcr}
\end{prop}
\begin{proof}

We will prove that $$u_\phi=[u_0 \wedge (\phi(u^+) \vee u^-) \vert_{t=0},u^-,\phi(u^+)\vee u^-]$$ has the same zero set as $u$. All the other equalities can be prove with a similar strategy.

We begin the proof assuming $\phi(x) \gs x.$ Then, $u_\phi = [u_0,u^-,\phi(u^+)].$\\
First, let us notice that the classical invariance for geometric equations proves immediately that $\phi(u)$ is the solution $[\phi(u_0),\phi(u^-),\phi(u^+)].$ In addition, thanks to Lemma \ref{compob} $u_\phi \gs u$ and $u_\phi \ls \phi(u).$ As a result, since $\{\phi(u) = 0\} = \{u=0\}$, we conclude that $\{u=0\} = \{u_\phi = 0\}$, what was expected.

Assume now that $\phi(x) \ls x$. The same arguments shows that $\phi(u) \ls u_\phi \ls u$, which leads to the same conclusion.

To conclude the proof for a general $\phi$, just introduce $f(x)=\min(x,\phi(x))$ and $g(x) = \max(x,\phi(x))$ and notice that since $\phi$ is nondecreasing, $\phi = f \circ g.$ So,
$$\{u=0\} = \{u_f = 0\} = \{(u_g)_f = 0\} = \{u_{f\circ g} = 0 \} = \{u_\phi = 0\}.$$
\end{proof}

Now, to be able to define a real geometrical evolution, we want a more general independence, which is contained in the following
\begin{thm}
Let $u = [u_0,u^-,u^+]$. Then, $\{u=0\} = \{v=0\}$ with $v=[v_0,v^-,v^+]$ under the (only) assumptions that 
$$\{u_0=0\} = \{v_0=0\},\quad \{u^- = 0\} = \{v^-=0\} \quad \text{and} \quad \{u^+ = 0\} = \{v^+=0\}.$$
\label{zeroset}
\end{thm}

\begin{proof}
This proof is based on the independence with no obstacles which is proved in \cite{evans91}, Theorem 5.1. We assume first that $u^-=v^-$ and $u^+=v^+$. As in \cite{evans91}, we define
$$\forall k \in \mathbb Z \setminus \{0\}, \quad E_k=\left\{x \in \Rn \, \middle \vert \, u_0(x) \gs \frac{1}{k} \right \}$$
and
$$a_k = \max_{\Rn \setminus E_k} v_0.$$
It is easy to see that
$$\forall k >0, \; a_1 \gs a_2 \gs \cdots \to 0 \quad \text{and} \quad a_{-1} \ls a_{-2} \ls \cdots \to 0.$$
Let us introduce $\phi : [-N,N] \to [-N,N]$ (with $N \gs \Vert u^\pm \Vert_\infty$, piecewise affine, by
$$\phi(\pm N) = \pm N,\quad \phi\left( \frac 1k \right) = a_k \quad \text{and} \quad \phi(0)=0.$$
Then, by definition, $\phi(u_0) \gs v_0$, $\{\phi = 0\} = \{0\}$ and $\phi$ is nondecreasing continuous. Thanks to Proposition \ref{invcr}, the solution $u_\phi := [\phi(u_0) \wedge u^+,u^-,u^+]$ has the same zero level-set as $u$, and is bigger than $v$ by comparison principle. Hence 
$$\{v \gs 0\} \subset \{u_\phi \gs 0\} = \{u\gs 0\}.$$

We prove the reverse inclusion switching $u_0$ and $v_0$.

Now, we assume that $u_0 = v_0$, $u^- = v^-$ and $u^+ \ls v^+.$ Then, by Lemma \ref{compob}, $u \ls v.$ We have just seen that there exists $\phi : [-N,N] \to [-N,N]$ nondecreasing continuous such that $\phi(u^+) \gs v^+$ and $\{\phi = 0 \} = \{0\}.$ Let $u_\phi=[u_0,u^-,\phi(u^+) \vee u^-].$ We saw that $u_\phi$ has the same zero set as $u$. In addition, by comparison, $u_\phi \gs v$. As a matter of fact,
$$\{u=0\} = \{ v=0\} = \{u_\phi = 0 \}.$$

If we drop the assumption $u^+ \ls v^+$, notice that $[u_0,u^-,u^+]$ and $[u_0,u^-,u^+ \wedge v^+]$ have the same zero level-set, so do $[u_0,u^-,v^+]$ and $[u_0,u^-,u^+\wedge v^+]$. Hence $[u_0,u^-,u^+]$ and $[u_0,u^-,v^+]$ have the same zero level-set.

Of course, changing only $u^-$ leads to the same result.

To show the general case, juste note that that $[u_0,u^-,u^+]$ and $[u_0,u^-,v^+]$ have the same zero level-set, so do $[u_0,u^-,v^+]$ and $[u_0,v^-,v^+]$, and $[u_0,v^-,v^+]$ and $[v_0,v^-,v^+]$  , and the first and the last ones.
\end{proof}
\subsection{Obstacles create fattening}
Although the fattening phenomenon may already occur without any obstacle (see \cite{bellettini98} for examples and \cite{barles93,biton08} for a more general discussion), obstacles will easily generate fattening whereas the free evolution is smooth. Consider $A$ a set of three points in $\R^2$ spanning an equilateral triangle and $S$ a circle enclosing it, centered on the triangle's center. Let $u^- = -1$, $u^+ = dist(\cdot,A)$, $u_0 = dist(\cdot, S)$ and $F(Du,D^2u) = - |D u| \div\left( \frac{D u}{|D u |} \right).$

It is possible to show (see next section) that the level sets $\{u(\cdot,t) \ls \alpha\}$ are minimizing hulls, hence are convex. So, the level set $\{u \ls 0\}$ contains the equilateral triangle. On the other hand, the level sets $\{u \ls -\delta\}$ behave as if there were no obstacles at all (in Proposition \ref{zeroset}, one can take $u^+ \equiv 1$ which has the same $-\delta$-set as $d(\cdot, A)$), so they disappear in finite time. As a result, $u=0$ in the whole triangle, and $\{u=0\}$ develops non empty interior.

\section{Comparison with a variational discrete scheme and long-time behavior}
\label{longtime}
In this section, we study the behavior of the mean curvature flow only\footnote{That means $u_t = |D u| \div \left( \frac{D u}{\vert D u \vert} \right)$.} with no forcing term and time independent obstacles, in large times. We assume moreover that $\Omega^+ = \R^n$ so that the obstacle is only from inside. For simplicity, we write $\Omega$ instead of $\Omega^-.$ In particular, we show that for relevant initial conditions ($E_0$ is assumed to be a \emph{minimizing hull}, see Definition \ref{defminh}), the flow has a limit.

In order to get some monotonicity properties of the flow, we will link our approach to a variational discrete flow built in \cite{spadaro11} and \cite{chambolle12} and inspired by \cite{ATW}. Starting from a set $E$ and an obstacle $\Omega \subset E$, these two papers introduce the following energy
\begin{equation} \mathcal E_h(E) = \min_{F \supset \Omega} \per(F) + \frac 1h \int_{F} d_{E}. \label{eq:energy}\end{equation}
In the previous energy, $\per (E)$ denotes the perimeter of the finite perimeter set $E$ (see \cite{giusti84} for an introduction to finite perimeter sets) and $d_E$ is the signed distance function to the set $E$ (positive outside $E$, negative inside).
\begin{rem}
Note that Spadaro introduces the energy
$$ \tilde {\mathcal E}_h(E) := \min_{{\Omega \subset F}} \left[ \operatorname{Per}(F) + \frac{1}{h} \int_{F \Delta E} \operatorname{dist}(x,\partial E) dx \right]. $$
One can see that it provides the same minimizers as \eqref{eq:energy} (not the same minimum, though). Indeed, one can write
$$ \operatorname{Per}(F) + \frac{1}{h} \int_{F \Delta E} \operatorname{dist}(x,\partial E) dx = \operatorname{Per}(F) + \frac{1}{h} \int_{F \setminus E} \operatorname{dist}(x,\partial E) dx + \frac{1}{h} \int_{E \setminus F} \operatorname{dist}(x,\partial E)$$
whereas
$$ \per(F) + \frac 1h \int_{F} d_{E} = \per (F) +\frac 1h \int_{F \setminus E} \operatorname{dist}(x,\partial E) dx - \frac 1h \int_{F \cap E} \operatorname{dist}(x,\partial E) dx.$$
Then, we can realize that the difference between the two energies is
  $$\frac{1}{h} \int_{E \setminus F} \operatorname{dist}(x,\partial E) + \frac 1h \int_{F \cap E} \operatorname{dist}(x,\partial E) dx =  \frac 1h \int_E \operatorname{dist}(x,\partial E) dx$$
which does not depend on $F$. Therefore, the two energies have the same minimizers.
\end{rem}
It has to be noticed that minimizers of these energies are not unique. 
\vspace{0.3cm}
To establish the comparison between these two approaches, we introduce
\begin{itemize}
 \item $u_0 : \R^n \to [-1,1]$ a uniformly continuous function such that $\{u_0 \ls 0\} = E_0$ (we make more assumptions later)
 \item $u^+ : \R^n \to [-1,1]$ a uniformly continuous function such that $\{u^+ \ls 0 \} = \Omega$ and $u^+ \gs u_0$. 
 \item $u^- = -1$.
\end{itemize}

In what follows, we will be interested in the 0-level-set of the solution u to 
$$u_t = |Du| \div \left( \frac{Du}{\vert Du \vert} \right)$$
with obstacles $u^\pm$ and initial condition $u_0$. More precisely, we want to show that for suitable $E_0$, the 0-level-set of the solution $\{u = 0\}$ converges to a minimal surface with obstacles.\\
We recall that thanks to Theorem \ref{zeroset}, any choice of $u_0$, $u^+$ satisfying the assumptions above will lead to the same evolution of the zero level-set of the solution.

\subsection{The discrete flow for sets}
Following \cite{spadaro11}, we define
\begin{defi}
$E$ is said to be a minimizing hull if $| \partial E| = 0$ (this is not assumed in the definition in \cite{spadaro11}, but is assumed stating minimizing hull properties) and
$$\per(E) \ls \per(F), \quad \forall F \supset E \; \text{with} \; F\setminus E \text{ compact}.$$
\label{defminh}
\end{defi}
Spadaro then shows the
\begin{prop}
Let $E$ be a minimizing hull. Then
\begin{itemize}
 \item For every $h>0$, one can define a (unique) maximal (with respect to $\subset$) minimizer in \eqref{eq:energy}, denoted in what follows by $T_h(E)$ (for every other minimizer $F$ of \eqref{eq:energy}, one has $F \subset T_h(E)$),
 \item $T_h(E) \subset E$ and $T_h(E)$ is still a minimizing hull (the measure of the boundary remains zero thanks to the classical regularity of minimizers (see for example Appendix B in \cite{spadaro11})
 \item If $F$ is another minimizing hull and $F \subset E$, then $T_h(F) \subset T_h(E).$
\end{itemize}
\end{prop}

Then, he defines the following scheme
\begin{equation} E_h(t) := T_h^{\left \lfloor{t/h}\right \rfloor }(E_0). \label{mindisc} \end{equation}

Let us state a couple of properties of the flow which will allow us to pass to the limit in $h$.

\begin{prop}
 Let $E$ be a minimizing hull and $h > \tilde h$. Then, $T_h(E) \subset T_{\tilde h} (E)$ almost everywhere.
\end{prop}
\begin{proof}
 Indeed, Let $F:= T_h(E)$ and $\tilde F := T_{\tilde h} (E)$. Since $E$ is a minimizing hull, $F,\tilde F \subset E$ so $d_E \ls 0$ on $F \cup \tilde F$. Using the very definition of $F$ and $\tilde F$, one can write
$$\per(F\cap \tilde F) + \frac{1}{h} \int_{F\cap \tilde F} d_E \gs \per (F) + \frac 1h \int_F d_E$$
$$\per(F \cup \tilde F) + \frac{1}{\tilde h} \int_{F\cup \tilde F} d_E \gs \per \tilde F + \frac{1}{\tilde h} \int_{\tilde F} d_E.$$

Summing, we get 
$$\per(F\cap \tilde F) + \per(F \cup \tilde F) + \frac{1}{h} \int_{F \cap \tilde F} d_E + \frac{1}{\tilde h} \int_{F \cup \tilde F} d_E \gs \per (F) + \per \tilde F + \frac{1}{h} \int_F d_E + \frac{1}{\tilde h} \int_{\tilde F} d_E.$$
Since $\per (F \cap \tilde F) + \per (F \cup \tilde F) \ls \per (F) + \per \tilde F$, one has
$$\frac 1h \int_{F \cap \tilde F} d_E + \frac{1}{\tilde h} \int_{F\cup \tilde F} d_E \gs \frac 1h \int_F d_E + \frac{1}{\tilde h} \int_{\tilde F} d_E,$$
which means
$$\frac{1}{\tilde h} \int_{F \setminus \tilde F} d_E \gs \frac{1}{h} \int_{F \setminus \tilde F} d_E,$$
hence
$$\int_{F\setminus \tilde F} d_E \left( \frac{1}{\tilde h} - \frac{1}{h} \right) \gs 0.$$
Then, since $|\partial E| = 0$, $|F \setminus \tilde F|$ = 0. 
\end{proof}

To pass to the limit in $h$, we will want to control the ``motion speed'' (see Proposition \ref{apsoluc}). To do so, we will need the two following propositions. First, we compare the constrained and the free motions.
\begin{prop}
Let $E$ be a minimizing hull containing $\Omega$. Let $E^f$ be the free evolution of $E$ ($E^f = T_h(E)$ with $\Omega = \emptyset$) and $E^c$ the regular evolution ($E^c$ is the maximal minimizer of \eqref{eq:energy}). Then, $E^f \cup \Omega \subset E^c$.
\label{compfree}
\end{prop}
\begin{proof}
 Using the definition of $E^f$ and $E^c$, one can write
\begin{equation}\per(E^f \cap E^c) + \int_{E^f \cap E^c} \frac{d_{E}}{h} \gs \per(E^f) + \int_{E^f} \frac{d_{E}}{h} \label{eq:inter} \end{equation}
\begin{equation} \per(E^f \cup E^c) + \int_{E^f \cup E^c} \frac{d_{E}}{h} \gs \per(E^c) + \int_{E^c} \frac{d_{E}}{h}. \label{eq:uni} \end{equation}

Summing and using $\per(E\cap F) + \per (E \cup F) \ls \per (E) + \per (F)$, we get
$$\int_{E^c \cap E^f} \frac{d_{E}}{h} + \int_{E^c \cup E^f} \frac{d_{E}}{h} \gs \int_{E^f} \frac{d_{E}}{h} + \int_{E^c} \frac{d_{E}}{h},$$
which is an equality. We conclude that \eqref{eq:inter} and \eqref{eq:uni} are equalities. In particular,
$$\per(E^f \cup E^c) + \int_{E^f \cup E^c} \frac{d_{E}}{h} = \per(E^c) + \int_{E^c} \frac{d_{E}}{h},$$which shows that $E^f \cup E^c$ is a minimizer of \eqref{eq:energy}. Since $E^c$ is the maximal minimizer, one has $E^f \subset E^c$.

One can also notice that by definition, $\Omega \subset T_h(E^c)$ so $T_h(E^f) \cup \Omega \subset T_h (E^c).$
\end{proof}

Then, it is easy to see that
\begin{itemize}
 \item A ball $B_R(x_0)$ is a minimizing hull,
\item For $h \ls \frac{R^2}{4n}$, we have $T_h(B_R(x_0)) = B_r(x_0)$ with $r = \frac{R + \sqrt{R^2 - 4nh}}2$.
\end{itemize}

Let us now show that $T_h$ preserves inclusion.
\begin{prop}
Let $\Omega^1 \subset \Omega^2$ be two obstacles and $E^1 \subset E^2$ be two minimizing hulls containing respectively $\Omega^1$ and $\Omega^2$. For $i \in \{1,2\}$, we introduce  
$$ E_h^i := \argmin_{E \supset \Omega^i} \per(E) + \frac{1}{h} \int_E d_{E^i},$$
where we choose $E_h^i$ to be maximal. Then, $E^1_h \subset E_h^2$.
\label{inclusion}
\end{prop}

\begin{proof}
Use the definition to write
\begin{equation}\per(E^1_h \cap E^2_h) + \int_{E^1_h \cap E^2_h} \frac{d_{E^1}}{h} \gs \per(E^1_h) + \int_{E^1_h} \frac{d_{E^1}}{h}, \label{eq:inter2} \end{equation}
\begin{equation} \per(E^1_h \cup E^2_h) + \int_{E^1_h \cup E^2_h} \frac{d_{E^2}}{h} \gs \per(E^2_h) + \int_{E^2_h} \frac{d_{E^2}}{h}, \label{eq:uni2} \end{equation}

Summing and simplifying, we get
$$\int_{E_h^1 \cap E^2_h} \frac{d_{E^1}}{h} + \int_{E_h^1 \cup E^2_h} \frac{d_{E^2}}{h} \gs \int_{E_1^h} \frac{d_{E^1}}{h} + \int_{E_2^h} \frac{d_{E^2}}{h}$$
which can be read
$$\int_{E^1_h \setminus E^2_h} \frac{d_{E^2}}{h} \gs \int_{E^1_h \setminus E^2_h} \frac{d_{E^1}}{h}$$ or again
$$\int_{E^1_h \setminus E^2_h} \frac{d_{E^1} - d_{E^2}}{h} \ls 0.$$

Since $E_1 \subset E_2$, one has $d_{E_2} \ls d_{E_1}$ which shows that the last inequality must in fact be an equality. As above, we conclude that showing as above that \eqref{eq:inter2} and \eqref{eq:uni2} are equalities, which proves that $E_h^1 \subset E_h^2.$
\end{proof}
Thanks to Propositions \ref{compfree} and \ref{inclusion}, one can conclude that the evolution $E_h$ of a minimizing hull $E_0$ contains the free evolution of every ball inside $E_0$.

\subsection{Passing to the limit}
Now, we want to define a similar iterative scheme but for the whole $u_0$. We assume that every level-set of $u_0$ is a minimizing hull ($E_0$ is assumed to be one and one can choose the other level sets of $u$ as we like to get this property).
\begin{rem}Starting from a minimizing hull $E_0$, it is easy to construct such a $u_0$. Let $\tilde u_0 $ be the signed distance function to $E_0$ truncated to $[-1,1]$. Let us define $u_0$ by replacing the level sets of $\tilde u_0$ $\tilde E_s := \{\tilde u_0 \ls s\}$ by $E_s$ the smallest (with respect to the inclusion) minimizer of $\per$ among the sets containing $\tilde E_s.$ \\
By definition, such sets must be minimizing hulls and the inclusion of the level sets is preserved so we can define $u_0$ by setting
$$\{ u_0 \ls s \} := E_s.$$

We now have to show that such a $u_0$ is continuous. If it were not, then there would exist $s<t$ and $x \in \overline{E_s^0} \cap \overline{(E_t^0)^c}$ (which reads formally $x \in \partial E_s^0 \cap E_t^0$). Since $\tilde u_0$ is continuous, the subset of such $x$ must be compact in $E_s \setminus \tilde E_s$ and $E_s \setminus \tilde E_t$. On the other hand the free boundaries $\partial E_\sigma \setminus \tilde E_\sigma$ for $\sigma \in [-1,1)$ have variational curvature zero (every small variation is admissible for the constraint $E_\sigma \supset \tilde E_\sigma$). We can then apply a cut and paste argument (see Th. 11 of \cite{Mer16} for a detailed proof) to show that this is not possible, and $u_0$ is therefore continuous.
\end{rem}
We define an evolution $u_h : \Rn \times [0,T[ \to [-1,1]$ by setting for all $s \in [-1,1]$, $E_s := \{u_0 \ls s\}$ and
$$\{u_h(t) \ls s\} = (E_s)_h (t).$$
This is well defined (in particular, $\{u_h(t) \ls s\} \subset \{u_h(t) \ls s'\}$ if $s \ls s'$) thanks to Proposition \ref{inclusion}.

One can easily notice that Proposition \ref{inclusion} gives the following monotonicity. If $u_0 \ls \tilde u_0$ are two functions whose level sets are minimizing hulls, $v \gs \tilde v$ two obstacle functions, then $u_h \ls \tilde u_h$.

Now, we want to pass to the limit in $h$ in the construction above. We will use the
\begin{prop}
If $u_0$ and $u^+$ are uniformly continuous (with modulus $\omega$), then the family $(u_h)$ is equicontinuous in space (with modulus $\omega$) and time.
\label{apsoluc}
\end{prop}
\begin{proof}
The arguments are standard and use the translation invariance of the scheme as well as the comparison principle.
 \begin{itemize}
\item{Space continuity.} The space continuity is easy to deduce. By continuity and translation invariance, $\tilde u_0(x) := u_0(x+z) \ls u_0(x) + \omega(|z|)$ and $\tilde u^+= u^+(\cdot +z) \ls u^+ + \omega(|z|)$ so $\tilde u_h \ls u_h + \omega(|z|)$, which was expected
\item{Time continuity.} Let $(x,t) \in \Rn \times \R^+$. Let $r >0.$ By uniform continuity in space, on $B_r(x)$, $u_h(\cdot,t) \ls u_h(x,t) + \omega(r),$ which means that $A^r:=\{u_h(\cdot, t) \ls u_h(x,t) + \omega(r)\}$ contains $B_r(x_0)$. Thanks to Proposition \ref{compfree}, the time evolution of $A^r$ contains the free evolution of $B_r(x_0)$, as long as the latter exists. That means $u_h(x,t+s) \ls u_h(x,t) + \omega(r)$ for $s \ls T_r$, extinction time of $B_r(x_0)$. It is easy to see that this time is controlled, for a sufficiently small $h$, by $\frac{r^2}{\sqrt{16h}}$. \\ 
We proved that for $h$ small enough, $u_h$ is continuous in time with modulus $\tilde \omega(T_r) \ls \omega(r).$
\end{itemize}
\end{proof}
\begin{cor}
 Up to a subsequence, the collection $(u_h)_h$ has a limit which is uniformly continuous in space and time. 
\end{cor}
Let us denote it by $u$ (we will see that this limit does not depend on the subsequence).

We are now able to show the main proposition of this section.
\begin{prop}
The function $u$ is the viscosity solution of \eqref{vsol}.
\label{limitscheme}
\end{prop}
\begin{proof}
This result is already known with no obstacles (one can directly apply \cite{chambolle122}, Th. 4.6 or, with a setting closer to ours, \cite{thouroude12}, Th 3.6.1. See also \cite{eto12}.) and could easily be adapted. Nonetheless, since our framework is simpler than \cite{chambolle122}, we give the whole proof here. We have just seen that $u$ is uniformly continuous in space and time. In addition, $u (t=0) = u_0$ by construction and the initial conditions are satisfied. We only have to check the fourth point of the definition (we only deal with the supersolution thing, the subsolution one can be treated similarly but is simpler because there is no real lower obstacle here). Let $(x,t) \in \Rn$. Either $u(x,t) = u^+(x,t)$ and nothing has to be done, or $u(x,t) < u^+(x,t)$. We proceed by contradiction and assume that there exists a smooth function $\varphi$ and $(\hat x, \hat t)$ such that $u-\varphi$ reaches a minimum at $(\hat x,\hat t)$ and that
\begin{equation} \left(\varphi_t - F^\ast(D\varphi, D^2 \varphi) \right) (\hat x, \hat t)< 0.\label{contrasovi}\end{equation}
One can assume that the minimum is strict and that $u- \varphi (\hat x, \hat t) = 0$. \\
First, we also assume that 
$$\nabla \varphi(\hat x, \hat t) \neq 0.$$
Thanks to an analogous of Proposition \ref{limjet}, one can find, for $h$ sufficiently small, $(x_h,t_h) \to (\hat x,\hat t)$ such that $u_h - \varphi$ reaches a minimum at $(x_h,t_h)$, $\nabla \varphi(x_h,t_h) \neq 0$, $u_h(x_h,t_h) < u^+(x_h,t_h)$ and $\left(\varphi_t - F(D\varphi, D^2 \varphi) \right) (x_h,t_h)< 0.$

Since $u_h -\varphi$ is minimal at $(x_h,t_h)$, we have
$$ E^h := \{x\, \vert \,u_h(x,t_h) \ls u_h(\hat x_h, \hat t_h) \} \subset \{x\, \vert \, \varphi(x,t_h) \ls \varphi(x_h, t_h)\} =: F. $$
Thanks to the minimum condition and continuity of $u_h$ and $\varphi$, we must have $x_h \in \partial E^h \cap \partial F$. In addition, $\nabla \varphi(x_h,t_h) \neq 0$ so $\partial F$ is a $\mathcal C^1$ graph around $x_h$. Recall finally that by construction, $E^h$ is some $E^n_h := T_h^{n}(E_0)$ with $n = [t_h/h]$ and therefore, minimizes $$\per(E) + \frac 1h \int_{E \Delta E_h^{n-1}} \left \vert d_{E_h^{n-1}} \right \vert.$$
Let $\nu_F =\frac{ \nabla \varphi }{|\nabla \varphi|}(x_h,t_h)$ be the unit vector normal to $F$ toward $F^c$ and consider 
$$F^\eps := F - \eps \nu$$
with $\eps$ sufficiently small such that $E^h \cap F^\eps$ is a compact perturbation of $E^h$ (from inside, see Figure \ref{pictlim}). 
\begin{figure}
\centering
\includegraphics[scale=0.4]{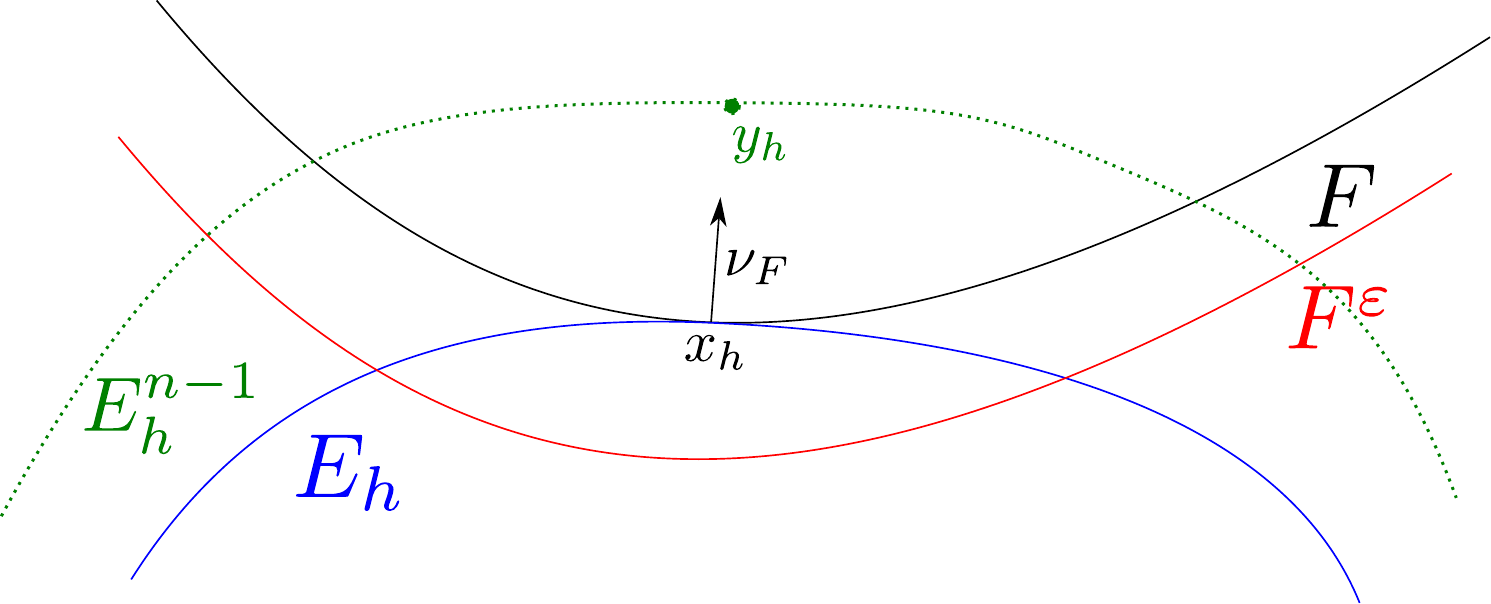}
\caption{Proof of Proposition \ref{limitscheme}}
\label{pictlim}
\end{figure}

This is possible since the minimum is strict. The minimizing property of $E^h$ can be written as
$$ \per(E_h^n) + \frac 1h \int_{E_h^n \Delta E_h^{n-1}} \left \vert d_{E_h^{n-1}} \right \vert \ls \per (E^h \cap F^\eps) + \frac 1h \int_{(E^h \cap F^\eps)  \Delta E_h^{n-1}} \left \vert d_{E_h^{n-1}} \right \vert.$$
Thus we have, recalling that the flow is monotone since we are dealing with minimizing hulls,
$$ \int_{ E_h^{n-1} \setminus E_h^n} \left \vert d_{E_h^{n-1}} \right \vert - \int_{E_h^{n-1} \setminus (E^h \cap F^\eps)  } \left \vert d_{E_h^{n-1}} \right \vert \ls h(\per (E^h \cap F^\eps) - \per(E^h)).$$
Now, let us notice that since $F^\eps$ is a smooth set, we have
$$ \per(E^h \cap F^\eps) = \per(E^h ; F^\eps) + \per( F^\eps; E^h)$$ so we can rewrite
\begin{equation}- \int_{E^h \setminus F^\eps} \left \vert d_{E_h^{n-1}} \right \vert \ls h(\per(F^\eps ; E^h) - \per(E^h ; (F^\eps)^c)).\label{eq:compar}\end{equation}
Finally, we get
$$ \int_{E^h \setminus F^\eps} \left \vert d_{E_h^{n-1}} \right \vert \gs h( \per(E^h ; (F^\eps)^c)- \per(F^\eps ; E^h)). $$
Observing that if $\nu^\eps$ is the outer normal vector to $F^\eps$,
$$\per(F^\eps ; E^h) = \int_{\partial F^\eps \cap E^h} 1 \ d\mathcal H^{n-1} = \int_{\partial F^\eps \cap E^h} \frac{\nabla \varphi}{|\nabla \varphi|} \cdot \nu^\eps \ d\mathcal H^{n-1}  $$
and if $\nu^h$ is the outer normal to $E^h$ and $\partial^\ast E^h$ its reduced boundary, we have
$$\per(E^h ; (F^\eps)^c) = \int_{\partial^\ast E^h \cap (F^\eps)^c} 1 \ d \mathcal H^{n-1} \gs \int_{\partial^\ast E^h \cap (F^\eps)^c} \frac{\nabla \varphi}{|\nabla \varphi|} \cdot \nu^h \ d \mathcal H^{n-1}.$$
Plugging into \eqref{eq:compar} and denoting by $\nu$ the outer normal vector to $E^h \setminus F^\eps$ ($\nu = \nu^h$ on $\partial E^h$ and $\nu = - \nu^\eps$ on $\partial F^\eps$) we have
$$ \int_{E^h \setminus F^\eps} \left \vert d_{E_h^{n-1}} \right \vert \gs h\int_{\partial^\ast (E^h \setminus F^\eps)} \frac{\nabla \varphi}{|\nabla \varphi|} \cdot \nu \ d \mathcal H^{n-1},$$
which, applying Green's formula, gives
$$\int_{E^h \setminus F^\eps} \left \vert d_{E_h^{n-1}} \right \vert \gs \int_{E^h \setminus F^\eps}  h\div \left( \frac{\nabla \varphi}{|\nabla \varphi|} \right) .$$
Letting $\eps$ go to zero, we get, at $(x_h,t_h)$,
\begin{equation}\left \vert d_{E_h^{n-1}} \right \vert \gs h\div \left( \frac{\nabla \varphi}{|\nabla \varphi|} \right). \label{eq:contrad} \end{equation}

Now, let $y_h \in \partial E_h^{n-1}$ which realizes the distance between $x_h$ and $(E_h^{n-1})^c$. By construction, we have
$$ u_h(y_h, t_h - h) = u_h(x_h,t_h)$$
So, since $(x_h,t_h)$ realizes the minimum of $u_h - \varphi$, we have 
$$\varphi(y_h,t_h-h) \ls \varphi(x_h,t_h).$$
Then, let us write
$$\varphi(y_h,t_h -h) = \varphi(x_h,t_h) -h \varphi_t (x_h,t_h) + \nabla \varphi (x_h,t_h) \cdot (y_h-x_h) +o(h+x_h-y_h), $$
we get
$$-h \varphi_t (x_h,t_h) + \nabla \varphi (x_h,t_h) \cdot (y_h-x_h) +o(h+x_h-y_h) \ls 0.$$
Since the level sets of $u_h$ are minimizing hulls, $u_h$ is non decreasing, which implies $\varphi_t \gs 0$. On the other hand, $\nabla \varphi(x_h,t_h)$ must point outside $E^h$ so $\nabla \varphi (x_h,t_h) \cdot (y_h-x_h) \gs 0$.
This implies
$$ |\nabla \varphi (x_h,t_h)| \left \vert d_{E_h^{n-1}} \right \vert \ls  h \varphi_t.$$
Replacing that into \eqref{eq:contrad}, we obtain, at $(x_h,t_h)$,
$$\varphi_t \gs |\nabla \varphi| \div\left(\frac{\nabla \varphi}{|\nabla \varphi|} \right).$$
Since $\varphi$ is smooth and $\nabla \varphi( \hat x , \hat t) \neq 0$; we can pass to the limit in $h$ and get a contradiction.

Let us now deal with the case $\nabla \varphi(\hat x, \hat t) = 0$ and consider the sequence $(x_h,t_h)$ constructed as before. Then, either one can find a subsequence $(x_{h_k},t_{h_k}) \to (x,t)$ such that $\nabla \varphi (x_{h_k},t_{h_k}) \neq 0$ or we have for every $h$ sufficiently small, $\nabla \varphi (x_h,t_h) =0.$ \\
In the first alternative, note that what we have just done still applies with minor changes. Indeed, we just have to get the contradiction taking the limsup instead of the full limit. The definition of $F^\ast$ ensures we keep the inequality. \\
On the other hand, if $\nabla \varphi (x_h,t_h) =0.$ for every small $h$, then we add a term $|x-\hat x|^\alpha$ (we denote by $\tilde \varphi$ the sum), with $\alpha > 2$, to $\varphi.$ The first and second derivative of $\varphi$ do not change. If one can find $\alpha$ such that $u_h - \tilde \varphi$ has a maximum at some $(x_\alpha^h,t_\alpha^h)$ with $\nabla \tilde \varphi (x_\alpha^{h_n},t_\alpha^{h_n}) \neq 0$ for a subsequence $h_n \to 0$, then we get the same contradiction. If not, that means that 
$$\forall \alpha > 2, \quad \nabla \varphi (x_\alpha^h,t_\alpha^h) = \alpha x_\alpha^h |x_\alpha^h - x_0|^{\alpha - 2}$$ for all $h$ sufficiently small, which imposes that $\varphi$, which is smooth, must have a non zero derivative of order $k \ls \alpha -1$ at $(\hat x,\hat t)$. This is not possible.

\end{proof}

\subsection{The time-limit is locally minimal}
We saw that since $u_0$ has minimizing hull level sets, so does $u_h(\cdot,t)$ and $u$ is therefore nondecreasing in time (this is true for $u_h$). As $u$ is uniformly equicontinuous on each compact set, letting $t$ go to $+\infty$ we have a locally uniform convergence to a limit $u_\infty$ which is a viscosity solution of 
$$|Du| \div \left( \frac{Du}{|Du|} \right) = 0$$ with obstacles $u^+,u^-$, thanks to classical theory of viscosity solutions.

Thanks to \cite{ilm98}, Theorem 3.10, one has the following result.
\begin{prop}
Let us assume that $\mathcal H^{n-1}(\{u=0\}) < \infty.$ Then, there exists a relatively open set $U \subset u^{-1}(s)$ with $H^{n-8-\alpha}(u^{-1}(0) \setminus U)=0$ for all $\alpha >0$, such that $u^{-1}(0) \setminus \Omega$ is an analytic minimal surface in a neighborhood of each point of $U$. Moreover, it is stable and stationary in the varifold sense (classically on $U$).
\label{limmin}
\end{prop}
Note in particular that non empty interior of $u^{-1}(s)$ can occur for only countable many $s$.
\subsection{Comparison with mean convex hull}
In \cite{spadaro11}, E. Spadaro is interested in the long time behavior of the discrete scheme \eqref{mindisc} but with a step $h$ which remains fixed. In this short subsection, we prove that if $\{u=0\}$ does not fatten, then our approach and Spadaro's build the same surface. The dimension of the ambient space $n$ is assumed to be less or equal to 7. Here are the theorems he gets:
\begin{thm}[Spadaro, \cite{spadaro11}]
Let $\Omega \subset \R^n$, $n \ls 7$, be a $\mathcal C^{1,1}$ closed set and $E_0 \supset \Omega$ a minimizing hull. Then, for a fixed $h$, the iterative scheme \eqref{mindisc} converges in time to some limit $E_\infty^h$. In addition, the $E_\infty^h$ converge monotonically to some $E_\infty$ which satisfies
\begin{itemize}
 \item $E_\infty$ is $\mathcal C^{1,1}$,
\item $E_\infty$ is a minimizing hull,
\item $\partial E_\infty \setminus \Omega$ is a (smooth) minimal surface.
\end{itemize}
\end{thm}
In addition, Spadaro uses this construction starting from $E_0$ with obstacles $\Omega_\varepsilon := \{x \in \R^n \, \vert \, d(x,\Omega) \ls \varepsilon\}$ to build a limit $E_\infty^\varepsilon.$ 
\begin{thm}[Spadaro]
 The set $$\Omega^{\text{mc}} := \bigcap_{\varepsilon >0} E_\infty^\varepsilon$$ is the mean convex hull of $\Omega$. That means
$$\Omega^{\text{mc}} = \bigcap_{\Omega \subset \Theta \in \mathcal A} \Theta$$
where $\mathcal A$ is the family of $\Theta \in \R^n$ such that for every minimal surface $\Sigma$ such that $\partial \Sigma \subset \Theta$, we have $\Sigma \subset \Theta.$
\end{thm} 
Let us show that $\Omega^{mc}$ agrees with our limit $\{u_\infty = 0\}.$
Since Spadaro's work is in low dimension, the open set $U$ in Proposition \ref{limmin} is the whole $u^{-1}(0)$. Let us assume that $u_\infty^{-1}(0)$ does not fatten. Hence, $\partial \{u_\infty\ls 0\} = \{u_\infty = 0 \}$ and $\{u_\infty= 0\} \setminus \Omega$ is a minimal hypersurface with boundary in $\Omega$. Using the very definition of the global barrier, we deduce that $\{u \ls 0 \} \subset \Omega^{mc}$.

Now, recalling that $\Omega^{mc}$ is a minimizing hull, it is in particular mean-convex, so if $v$ is the truncated signed distance function to $\Omega^{mc}$, it is a stationary subsolution of \eqref{motion}. Let us prove that it is also a supersolution. We know that $\partial \Omega^{mc}$ is a minimal surface out of the obstacle, so $v$ satisfies
$$-|\nabla v| \div \left( \frac{\nabla v}{|\nabla v|}\right) = 0$$ in the classical sense whenever $v < u^+.$ That is exactly saying that $v$ is a supersolution of \eqref{motion}.

Then, the comparison principle (Proposition \ref{compple}) implies, since $v \ls u_0$, that $v \ls u$ and then $\{u \ls 0\} \supset \Omega^{mc}$.

Finally, $$\{u \ls 0\} = \Omega^{mc}$$ and both approaches coincide.

\subsection*{Acknowledgment}
I am grateful to Antonin Chambolle for introducing me to this problem, and for fruitful discussions. I would also like to thank Matteo Novaga for the suggestions he made, especially concerning the good framework to tackle this problem, and his interest in this work.\\
This work was partially supported by the ANR (Agence Nationale de la Recherche) through HJnet project ANR-12-BS01-0008-01.

\end{document}